\documentclass{amsart}

\usepackage{amsmath,amssymb,amsthm}
\usepackage[all]{xy}
\usepackage[colorlinks=true]{hyperref}

\numberwithin{equation}{section}

\SelectTips{eu}{12}

\newtheorem{thm}{Theorem}[section]
\newtheorem{prop}[thm]{Proposition}
\newtheorem{lem}[thm]{Lemma}

\theoremstyle{definition}

\theoremstyle{remark}
\newtheorem{rem}[thm]{Remark}

\newcommand{\ZZ}{\mathbb{Z}}

\newcommand{\st}{{\rm st}}

\title[Matching complexes of small grids]{Matching complexes of small grids}

\author{Takahiro Matsushita}
\address{Department of Mathematical Sciences, University of the Ryukyus, Nishihara-cho, Okinawa 903-0213, Japan}
\email{mtst@sci.u-ryukyu.ac.jp}


\begin{document}

\baselineskip.525cm

\maketitle

\begin{abstract}
The matching complex $M(G)$ of a simple graph $G$ is the simplicial complex consisting of the matchings on $G$. The matching complex $M(G)$ is isomorphic to the independence complex of the line graph $L(G)$.

Braun and Hough introduced a family of graphs $\Delta^m_n$, which is a generalization of the line graph of the $(n \times 2)$-grid graph. In this paper, we show that the independence complex of $\Delta^m_n$ is a wedge of spheres. This gives an answer to a problem suggested by Braun and Hough.
\end{abstract}

\section{Introduction}

A {\it matching} on a simple graph $G = (V(G), E(G))$ is a subgraph of $G$ whose maximal degree is at most 1. A matching is identified with its edge set. The {\it matching complex $M(G)$ of $G$} is the simplicial complex whose simplices are the matchings on $G$. We refer to \cite{Jonsson2} for a concrete introduction to this subject.

In this paper, we study the homotopy types of the matching complexes of the $(n \times 2)$-grid graphs. For a pair $m$ and $n$ of positive integers, the {\it $(m \times n)$-grid graph $\Gamma(m,n)$} is defined by
$$V(\Gamma(m,n)) = \{ (i,j) \in \ZZ^2 \; | \; 1 \le i \le m, \; 1 \le j \le n\},$$
$$E(\Gamma(m,n)) = \{ \{ (i,j), (i', j')\} \; | \; |i' - i| + |j' - j| = 1\}.$$
In particular, we write $\Gamma_n$ instead of $\Gamma(n,2)$.

Kozlov \cite{Kozlov1} showed that the matching complex of $\Gamma(n,1)$ is contractible or homotopy equivalent to a sphere. However, the topology of a matching complex is in general very complicated, even for simple examples of graphs. For example, the matching complexes of complete graphs and complete bipartite graphs have torsions in their integral homology groups (see \cite{BLVZ}, \cite{Jonsson2}, and \cite{SW}). After Jonsson's unpublished work \cite{Jonsson1} concerning the matching complexes of general grid graphs, Braun and Hough \cite{BH} investigate the matching complex of $\Gamma_n$, and wrote ``the topology of the matching complex for the $2 \times n$ grid graph is quite mysterious''. However, in this paper we determine the homotopy type of the matching complex of $\Gamma_n$ completely, and show that they are wedges of spheres. In fact, we determine the homotopy types of independence complexes of some family of graphs $\Delta^m_n$ introduced by Braun and Hough \cite{BH}. To state it precisely, we need some preparation.

For a graph $G$, the {\it independence complex $I(G)$ of $G$} is the simplicial complex whose simplices are the independent sets of $G$. The {\it line graph $L(G)$ of $G$} is the graph whose vertex set is the edge set $E(G)$ of $G$, and two distinct edges $e$ and $e'$ of $G$ are adjacent if and only if they have a common endpoint. Then the matching complex $M(G)$ coincides with the independence complex of the line graph $L(G)$. Figure 1 depicts the line graph of $\Gamma_5$. Here $e_i$, $f_i$, and $f'_i$ denote the edges $\{ (1,i), (2,i)\}$, $\{ (i,1), (i+1,1)\}$, and $\{ (i,2), (i+1,2)\}$ of $\Gamma_n$, respectively.

\begin{figure}[t]
\begin{center}
\begin{picture}(180,90)(0,-5)
\multiput(10,40)(40,0){5}{\circle*{3}}
\multiput(30,70)(40,0){4}{\circle*{3}}
\multiput(30,10)(40,0){4}{\circle*{3}}

\multiput(10,40)(40,0){4}{\line(2,3){20}}
\multiput(10,40)(40,0){4}{\line(2,-3){20}}
\multiput(50,40)(40,0){4}{\line(-2,3){20}}
\multiput(50,40)(40,0){4}{\line(-2,-3){20}}

\multiput(30,10)(0,60){2}{\line(1,0){120}}

\put(-3,38){\tiny $e_1$} \put(175,38){\tiny $e_5$}
\put(155,2){\tiny $f'_{4}$} \put(153,74){\tiny $f_{4}$}
\put(19,2){\tiny $f'_1$} \put(19,74){\tiny $f_1$}

\end{picture}

{\bf Figure 1}
\end{center}
\end{figure}

\begin{figure}[b]
\begin{center}
\begin{picture}(260,100)(0,0)
\multiput(40,10)(0,30){4}{\circle*{3}}
\multiput(100,10)(0,30){4}{\circle*{3}}
\multiput(160,10)(0,30){4}{\circle*{3}}
\multiput(220,10)(0,30){4}{\circle*{3}}

\multiput(40,10)(0,30){4}{\line(1,0){180}}

\multiput(10,55)(60,0){5}{\circle*{3}}

\multiput(10,55)(60,0){4}{\line(2,3){30}}
\multiput(10,55)(60,0){4}{\line(2,1){30}}
\multiput(10,55)(60,0){4}{\line(2,-1){30}}
\multiput(10,55)(60,0){4}{\line(2,-3){30}}

\multiput(70,55)(60,0){4}{\line(-2,3){30}}
\multiput(70,55)(60,0){4}{\line(-2,1){30}}
\multiput(70,55)(60,0){4}{\line(-2,-1){30}}
\multiput(70,55)(60,0){4}{\line(-2,-3){30}}

\put(260,52){$\Delta^4_5$}
\end{picture}

{\bf Figure 2.}
\end{center}
\end{figure}

For a pair $m$ and $n$ of positive integers, Braun and Hough \cite{BH} introduced the graph $\Delta^m_n$, which is a generalization of $L(\Gamma_n)$. The vertex set of $\Delta^m_n$ consists of $e_i$ for $i = 1, \cdots, n$ and $f^k_i$ for $i = 1, \cdots, n-1$ and $k = 1, \cdots, m$. The adjacent relations are given as follows:
$$f^k_i \sim f^k_{i+1}, (i= 1, \cdots, n-2), \; e_i \sim f^k_i \sim e_{i+1}, (i = 1, \cdots, n-1)$$
Figure 2 depicts the graph $\Delta^4_5$. Clearly, $\Delta^2_n$ and $L(\Gamma_n)$ are isomorphic, and hence $I(\Delta^2_n)$ and $M(\Gamma_n)$ are isomorphic.

Braun and Hough \cite{BH} actually studied\footnote{Our definition of $\Delta^m_n$ is a little different from the one of \cite{BH}. Namely, their $\Delta^m_n$ is our $\Delta^m_{n+2}$.} the independence complexes of $\Delta^m_n$. The purpose of this paper is to determine the homotopy types of the independence complexes of $\Delta^m_n$. The following two theorems are the main results in this paper.

\begin{thm} \label{thm 1.0}
$\Delta^1_{2n} \simeq S^{n-1}$ and $\Delta^1_{2n-1} \simeq *$ for $n \ge 1$.
\end{thm}

\begin{thm} \label{thm 1.1}
For $n \ge 5$ and $m \ge 2$, we have
$$I(\Delta^m_n) \simeq \Sigma^2 I(\Delta^m_{n-3}) \vee \Sigma^m I(\Delta^m_{n-3}) \vee \Sigma^{m+1} (\Delta^m_{n-4}).$$
Here $\Sigma$ denotes the reduced suspension.
\end{thm}

\begin{rem}
The equation among the Euler characteristics of $I(\Delta^m_n)$ obtained by Theorem \ref{thm 1.1} is known. See Corollary 16 of \cite{BH}.
\end{rem}

In particular, we have
$$M(\Gamma_n) \simeq \Sigma^2 M(\Gamma_{n-3}) \vee \Sigma^2 M(\Gamma_{n-3}) \vee \Sigma^3 M(\Gamma_{n-4}).$$

By Theorem \ref{thm 1.1}, the homotopy type of $I(\Delta^m_n)$ is determined by $I(\Delta^m_1), \cdots, I(\Delta^m_4)$ recursively. In Section 4, we determine the homotopy types of these complexes as follows:

\begin{prop} \label{prop 1.2}
For $m\ge 2$, the complexes $I(\Delta^m_1), \cdots, I(\Delta^m_4)$ are described as follows:
$$I(\Delta^m_1) = *, I(\Delta^m_2) \simeq S^0, I(\Delta^m_3) \simeq S^1 \vee S^{m-1}, I(\Delta^m_4) \simeq S^m$$
\end{prop}

Combining Theorem \ref{thm 1.1} and Proposition \ref{prop 1.2}, we have that the independence complex of $\Delta^m_n$ is a wedge of spheres. In particular, the integral homology groups of them have no torsions. This gives an answer to a problem suggested by Braun and Hough (see the end of \cite{BH}).

%\begin{thm} \label{thm 1.3}
%The independence complex $I(\Delta^m_n)$ of $\Delta^m_n$ is a wedge of spheres for every pair $m$ and $n$ of positive integers. In particular, the matching complex $M(\Gamma_n)$ of $\Gamma_n$ is homotopy equivalent to a wedge of spheres.
%\end{thm}

This paper is organized as follows. In Section 2, we review some facts concerning independence complexes. Since Theorem \ref{thm 1.0} is easily deduced from known results, we discuss it in this section. Theorem \ref{thm 1.1} and Proposition \ref{prop 1.2} are proved in Section 3 and Section 4, respectively.

\section{Preliminaries}

We refer to \cite{Jonsson2} and \cite{Kozlov} for fundamental terms and facts concerning simplicial complexes.

For a vertex $v$ of a simple graph $G$, let $N_G(v)$ denote the set of vertices adjacent to $v$. We write $N_G[v]$ to mean $N_G(v) \cup \{ v\}$. For a subset $S$ of $V(G)$, the subgraph of $G$ induced by $V(G) \setminus S$ is denoted by $G \setminus S$. In particular, we write $G \setminus v$ instead of $G \setminus \{ v\}$.

We first recall the following simple observation of independence complexes (see Adamaszek \cite{Adamaszek1}). For a vertex $v$ of $G$, the link of $v$ in $I(G)$ coincides with $I(G \setminus N_G[v])$. Since $I(G) \setminus v = I(G \setminus v)$, we have that $I(G)$ is the mapping cone of the inclusion $I(G \setminus N_G(v)) \hookrightarrow I(G \setminus v)$. Here $I(G) \setminus v$ denotes the subcomplex of $I(G)$ whose simplices are the simplices of $I(G)$ not containing $v$. This observation clearly yields the following proposition:

\begin{figure}[t]
\begin{center}
\begin{picture}(340,60)(0,-10)
\multiput(10,40)(40,0){4}{\circle*{3}}
\multiput(30,10)(40,0){4}{\circle*{3}}
\multiput(10,40)(40,0){4}{\line(2,-3){20}}
\multiput(30,10)(40,0){3}{\line(2,3){20}}
\put(10,40){\line(1,0){120}}
\put(10,40){\line(-1,0){10}} \put(10,40){\line(-2,-3){10}}

\put(148,0){\tiny $e_n$} \put(108,0){\tiny $e_{n-1}$}
\put(127,47){\tiny $f_{n-1}^1$} \put(87,47){\tiny $f_{n-2}^1$}

\put(160,23){$\simeq$}

\multiput(190,40)(40,0){2}{\circle*{3}}
\multiput(210,10)(40,0){2}{\circle*{3}}
\put(270,40){\circle{3}} \put(290,10){\circle{3}}
\put(310,40){\circle*{3}} \put(330,10){\circle*{3}}

\put(310,40){\line(2,-3){20}}

\multiput(190,40)(40,0){2}{\line(2,-3){20}}
\put(210,10){\line(2,3){20}}

\put(190,40){\line(1,0){40}}
\put(190,40){\line(-1,0){10}} \put(190,40){\line(-2,-3){10}}

\put(328,0){\tiny $e_n$} \put(307,47){\tiny $f_{n-1}^1$}
\put(248,0){\tiny $e_{n-2}$}
\end{picture}

{\bf Figure 3}
\end{center}
\end{figure}

\begin{prop}[See \cite{Adamaszek1}] \label{prop 2.3}
Let $v$ be a vertex of a graph $G$. If the inclusion $I(G \setminus N_G[v]) \hookrightarrow I(G \setminus v)$ is null-homotopic, then we have
$$I(G) \simeq I(G \setminus v) \vee \Sigma I(G \setminus N_G[v]).$$
\end{prop}

\begin{prop}[Lemma 2.5 of \cite{Engstrom}] \label{prop Engstrom}
Let $v$ and $w$ be a pair of distinct vertices of $G$ with $N_G(v) \subset N_G(w)$. Then the inclusion $I(G \setminus w) \hookrightarrow I(G)$ is a homotopy equivalence.
\end{prop}
\begin{proof}
By the above observation, it suffices to see that $I(G \setminus N_G[w])$ is contractible. But this is clear since $G \setminus N_G[w]$ has an isolated vertex $v$.
\end{proof}

Here we give the proof of Theorem \ref{thm 1.0} since it easily follows from Proposition \ref{prop Engstrom}.

\begin{prop} \label{prop 2.1}
If $n \ge 3$, then $I(\Delta^1_n) \simeq \Sigma I(\Delta^1_{n-2})$.
\end{prop}
\begin{proof}
Since $N_{\Delta^1_n}(e_n) \subset N_{\Delta^1_n}(e_{n-1})$ and $N_{\Delta^1_n}(e_n) \subset N_{\Delta^1_n}(f^1_{n-2})$ (see Figure 3), we have
$$I(\Delta^1_n) \simeq I(\Delta^1_n \setminus \{ e_{n-1}, f^1_{n-2}\}) = I(\Delta^1_{n-2}) * I(K_2) = \Sigma I(\Delta^1_{n-2}).$$
\end{proof}

\noindent
{\it Proof of Theorem \ref{thm 1.0}.}
It is clear that $I(\Delta^1_1) = *$ and $I(\Delta^1_2) = I(P_3) \simeq S^0$. Here $P_3$ denotes the path graph with 3-vertices. Thus Proposition \ref{prop 2.1} implies Theorem \ref{thm 1.0}.\qed

\section{Theorem \ref{thm 1.1}}

The purpose of this section is to prove Theorem \ref{thm 1.1}. Throughout this section, we assume that $m$ is an integer greater than 1. Suppose $n \ge 2$, and put $X_n = \Delta^m_n \setminus e_{n-1}$. Since $N_{\Delta^m_n}(e_n) \subset N_{\Delta^m_n}(e_{n-1})$, Proposition \ref{prop Engstrom} implies the following:

\begin{lem} \label{lem 3.1}
For $n\ge 2$ and $m \ge 2$, we have $I(\Delta^m_n) \simeq I(X_n)$.
\end{lem}

Next we consider the graph $Y_n = X_n \setminus e_{n-2}$ (see Figure 5).

\begin{figure}[b]
\begin{center}
\begin{picture}(360,120)(0,-15)
\multiput(10,10)(0,30){4}{\circle*{3}}
\multiput(70,10)(0,30){4}{\circle*{3}}
\multiput(130,10)(0,30){4}{\circle*{3}}
\multiput(40,55)(60,0){3}{\circle*{3}}

\put(106,53){\tiny $e_{n-1}$}

\multiput(10,10)(0,30){4}{\line(1,0){120}}
\multiput(10,10)(0,30){4}{\line(-1,0){20}}

\multiput(40,55)(60,0){2}{\line(2,3){30}}
\multiput(40,55)(60,0){2}{\line(2,1){30}}
\multiput(40,55)(60,0){2}{\line(2,-1){30}}
\multiput(40,55)(60,0){2}{\line(2,-3){30}}

\multiput(40,55)(60,0){3}{\line(-2,3){30}}
\multiput(40,55)(60,0){3}{\line(-2,1){30}}
\multiput(40,55)(60,0){3}{\line(-2,-1){30}}
\multiput(40,55)(60,0){3}{\line(-2,-3){30}}

\put(65,-10){$\Delta^m_n$}

\put(175,50){$\simeq$}

\multiput(220,10)(0,30){4}{\circle*{3}}
\multiput(280,10)(0,30){4}{\circle*{3}}
\multiput(340,10)(0,30){4}{\circle*{3}}
\put(250,55){\circle*{3}}
\put(310,55){\circle{3}}
\put(370,55){\circle*{3}}

\multiput(220,10)(0,30){4}{\line(1,0){120}}
\multiput(220,10)(0,30){4}{\line(-1,0){20}}

\put(250,55){\line(2,3){30}}
\put(250,55){\line(2,1){30}}
\put(250,55){\line(2,-1){30}}
\put(250,55){\line(2,-3){30}}

\put(250,55){\line(-2,3){30}}
\put(250,55){\line(-2,1){30}}
\put(250,55){\line(-2,-1){30}}
\put(250,55){\line(-2,-3){30}}

\put(370,55){\line(-2,3){30}}
\put(370,55){\line(-2,1){30}}
\put(370,55){\line(-2,-1){30}}
\put(370,55){\line(-2,-3){30}}

\put(275,-10){$X_n$}

\end{picture}

{\bf Figure 4.}
\end{center}
\end{figure}

\begin{figure}[t]
\begin{center}
\begin{picture}(340,90)(0,-10)
\put(30,40){\circle*{2}}
\put(150,40){\circle*{2}}
\put(70,40){\circle{2}}
\put(110,40){\circle{2}}
\multiput(10,10)(40,0){4}{\circle*{2}}
\multiput(10,30)(40,0){4}{\circle*{2}}
\multiput(10,50)(40,0){4}{\circle*{2}}
\multiput(10,70)(40,0){4}{\circle*{2}}

\put(30,40){\line(2,3){20}}
\put(30,40){\line(2,1){20}}
\put(30,40){\line(2,-3){20}}
\put(30,40){\line(2,-1){20}}

\put(30,40){\line(-2,3){20}}
\put(30,40){\line(-2,1){20}}
\put(30,40){\line(-2,-3){20}}
\put(30,40){\line(-2,-1){20}}

\put(150,40){\line(-2,3){20}}
\put(150,40){\line(-2,1){20}}
\put(150,40){\line(-2,-3){20}}
\put(150,40){\line(-2,-1){20}}

\multiput(10,10)(0,20){4}{\line(1,0){120}}
\multiput(10,10)(0,20){4}{\line(-1,0){10}}

\put(155,38){\tiny $e_n$}

\put(210,40){\circle*{2}}
\put(250,40){\circle{2}}
\put(290,40){\circle{2}}
\put(330,40){\circle*{2}}

\put(210,40){\line(-2,3){20}}
\put(210,40){\line(-2,1){20}}
\put(210,40){\line(-2,-3){20}}
\put(210,40){\line(-2,-1){20}}

\put(330,40){\line(-2,3){20}}
\put(330,40){\line(-2,1){20}}
\put(330,40){\line(-2,-1){20}}
\put(330,40){\line(-2,-3){20}}

\multiput(190,10)(0,20){4}{\circle*{2}}
\multiput(230,10)(0,20){4}{\circle{2}}
\multiput(270,10)(0,20){4}{\circle{2}}
\multiput(310,10)(0,20){4}{\circle*{2}}

\put(215,38){\tiny $e_{n-3}$}
\put(335,38){\tiny $e_n$}
\put(308,78){\tiny $f_{n-1}^1$}
\put(268,78){\tiny $f_{n-2}^1$}

\multiput(190,10)(0,20){4}{\line(-1,0){10}}

\put(66,-7){\footnotesize $Y_n$}
\put(224,-7){\footnotesize $X_n \setminus N[e_{n-2}]$}
\end{picture}

{\bf Figure 5}
\end{center}
\end{figure}

\begin{prop} \label{prop 3.2}
For $n \ge 4$ and $m\ge 2$, we have $I(X_n) \simeq I(Y_n) \vee \Sigma^2 I(\Delta^m_{n-3})$.
\end{prop}
\begin{proof}
We want to apply Proposition \ref{prop 2.3} to the vertex $e_{n-2}$ of $X_n$. Thus we need to show that $I(X_n \setminus N_{X_n}[e_{n-2}]) \simeq \Sigma I(\Delta^m_{n-3})$ and the inclusion $I(X_n \setminus N_{X_n}[e_{n-2}]) \hookrightarrow I(X_n \setminus e_{n-2}) = I(Y_n)$ is null-homotopic.

By Figure 5 and Proposition \ref{prop Engstrom}, it is clear that $I(X_n \setminus N_{X_n}[e_{n-2}]) \simeq I(\Delta^m_{n-3} \sqcup K_2) = \Sigma I(\Delta^m_{n-3})$. To see that the inclusion $I(X_n \setminus N_{X_n}[e_{n-2}]) \hookrightarrow I(X_n \setminus e_{n-2})$ is null-homotopic, we first see that the inclusion
$$I(X_n \setminus ( N_{X_n}[e_{n-2}] \cup \{ f^1_{n-1}\})) \hookrightarrow I(X_n \setminus e_{n-2}) = I(Y_n)$$ is a homotopy equivalence. Note that every vertex of $X_n \setminus (N_{X_n}[e_{n-2}] \cup \{ f^1_{n-1}\})$ is not adjacent to $f^1_{n-2}$ in $Y_n$. Thus $I(X_n \setminus ( N_{X_n}[e_{n-2}] \cup \{ f^1_{n-1}\}))$ is contained in the star $\st_{I(Y_n)}(f^1_{n-2})$. This means that the composite
$$I \big( X_n \setminus (N_{X_n}[e_{n-2}] \cup \{ f^1_{n-1}\}) \big) \xrightarrow{\simeq} I\big( X_n \setminus N_{X_n}[e_{n-2}] \big) \to I(Y_n)$$
is null-homotopic. It follows from Proposition \ref{prop Engstrom} that the first inclusion is a homotopy equivalence (Here we use the assumption $m \ge 2$). Thus the inclusion $I(X_n \setminus N_{X_n}[e_{n-2}]) \to I(Y_n)$ is null-homotopic, and this completes the proof.
\end{proof}

Finally we study the homotopy type of $I(Y_n)$

\begin{prop} \label{prop 3.3}
For $n \ge 5$ and $m\ge 2$, we have $I(Y_n) \simeq \Sigma^mI(\Delta^m_{n-3}) \vee \Sigma^{m+1} I(\Delta^m_{n-4})$
\end{prop}
\begin{proof}
We want to apply Proposition \ref{prop 2.3} to the vertex $e_n$ of $Y_n$. Namely, we must show the following:
\begin{itemize}
\item[(1)] The inclusion $I(Y_n \setminus N_{Y_n}[e_n]) \hookrightarrow I(Y_n \setminus e_n)$ is null-homotopic.

\item[(2)] The homotopy type of $I(Y_n \setminus N_{Y_n}[e_n])$ is $\Sigma^m I(\Delta^m_{n-4})$.

\item[(3)] The homotopy type of $I(Y_n \setminus e_n)$ is $\Sigma^m I(\Delta^m_{n-3})$.
\end{itemize}

Define the induced subgraphs $Z_n$, $Z_n'$, and $Z_n''$ of $Y_n$ as follows:
$$Z_n = Y_n \setminus \big( \{ f^i_{n-4} \; | \; i = 1, \cdots, m\} \cup \{ e_{n-3} \} \cup N_{Y_n}[e_n] \big),$$
$$Z'_n = Y_n \setminus \big( \{ f^i_{n-4} \; | \; i = 1, \cdots, m\} \cup \{ e_{n-3}, e_n \} \big),$$
$$Z''_n = Y_n \setminus (N_{Y_n}[e_{n-3}] \cup \{ e_n\}) $$
Figure 6 depicts the graphs $Z_n$, $Z'_n$, and $Z''_n$ in the case $m = 4$.

By Proposition \ref{prop Engstrom}, $I(Y_n \setminus N[e_n])$ is homotopy equivalent to $I(Z_n)$. Clearly, we have $I(Z_n) \simeq \Sigma^m I(\Delta^m_{n-4})$, which implies (2). By Proposition \ref{prop Engstrom}, the inclusions $I(Z_n) \hookrightarrow I(Z'_n)$ and $I(Z''_n) \hookrightarrow I(Z'_n)$ are homotopy equivalences. Since $I(Z''_n)$ is contained in the star $\st_{I(Y_n \setminus e_n)}(e_{n-3})$, we have that the inclusion $I(Z''_n) \hookrightarrow I(Y_n \setminus e_n)$ is null-homotopic. It follows from the commutative diagram
$$\xymatrix{
I(Z_n) \ar[r]^\simeq \ar[rd] & I(Z'_n) \ar[d] & I(Z''_n) \ar[l]_{\simeq} \ar[ld] \\
{} & I(Y_n \setminus e_n) & {}
}$$
that the inclusion $I(Z_n) \hookrightarrow I(Y_n \setminus e_n)$ is null-homotopic. By the sequence
$$I(Z_n) \xrightarrow{\simeq} I(Y_n \setminus N_{Y_n}[e_n]) \to I(Y_n \setminus e_n),$$
of inclusions, we have that $I(Y_n \setminus N_{Y_n}[e_n]) \hookrightarrow I(Y_n \setminus e_n)$ is null-homotopic. This completes the proof of (1).

Finally, we prove (3). By Proposition \ref{prop Engstrom}, it is easy to see that $I(Y_n \setminus e_n)$ is homotopy equivalent to $I(W_n)$ (see Figure 6). Here $W_n$ is defined by
$$W_n = Y_n \setminus ( \{ f^k_{n-3} \; | \; k = 1, \cdots, m\}\cup \{ e_n\}). $$
Clearly, $I(W_n)$ is homotopy equivalent to $\Sigma^m I(\Delta^m_{n-3})$. This completes the proof of (3).
\end{proof}

Combining Lemma \ref{lem 3.1}, Proposition \ref{prop 3.2}, and Proposition \ref{prop 3.3}, we have
$$I(\Delta^m_n) \simeq I(X_n) \simeq I(Y_n) \vee \Sigma^2 I(\Delta^m_{n-3})
\simeq \Sigma^m I(\Delta^m_{n-3}) \vee \Sigma^{m+1}(\Delta^m_{n-4}) \vee I(\Delta^m_{n-3}).$$
This completes the proof of Theorem \ref{thm 1.1}.

\begin{figure}[t]
\begin{center}
\begin{picture}(340,80)(0,0)
\multiput(10,10)(0,20){4}{\circle*{2}}
\multiput(10,10)(0,20){4}{\line(-1,0){15}}
\put(30,40){\circle*{2}}
\put(30,40){\line(-2,-3){20}}
\put(30,40){\line(-2,-1){20}}
\put(30,40){\line(-2,1){20}}
\put(30,40){\line(-2,3){20}}

\put(70,40){\circle{2}}

\put(34,39){\tiny $e_{n-4}$}

\put(78,77){\tiny $f^1_{n-3}$}
\put(123,77){\tiny $f^1_{n-1}$}
\put(154,38){\tiny $e_n$}

\multiput(50,10)(0,20){4}{\circle{2}}
\multiput(90,10)(0,20){4}{\circle*{2}}
\multiput(110,10)(0,20){4}{\circle*{2}}
\multiput(90,10)(0,20){4}{\line(1,0){20}}

\multiput(130,10)(0,20){4}{\circle{2}}
\put(150,40){\circle{2}}

%\put(45,52){\tiny $e_{n-4}$}

\put(66,-7){\footnotesize $Z_n$}

\multiput(210,10)(0,20){4}{\circle*{2}}
\multiput(210,10)(0,20){4}{\line(-1,0){15}}
\put(230,40){\circle*{2}}
\put(230,40){\line(-2,-3){20}}
\put(230,40){\line(-2,-1){20}}
\put(230,40){\line(-2,1){20}}
\put(230,40){\line(-2,3){20}}

\put(270,40){\circle{2}}

\multiput(250,10)(0,20){4}{\circle{2}}
\multiput(290,10)(0,20){4}{\circle*{2}}
\multiput(310,10)(0,20){4}{\circle*{2}}
\multiput(290,10)(0,20){4}{\line(1,0){40}}
\multiput(330,10)(0,20){4}{\circle*{2}}
\put(350,40){\circle{2}}

\put(266,-7){\footnotesize $Z'_n$}
\end{picture}

\begin{picture}(340,115)(0,-10)
\multiput(10,10)(0,20){4}{\circle*{2}}
\multiput(10,10)(0,20){4}{\line(-1,0){15}}
\put(30,40){\circle*{2}}
\put(30,40){\line(-2,-3){20}}
\put(30,40){\line(-2,-1){20}}
\put(30,40){\line(-2,1){20}}
\put(30,40){\line(-2,3){20}}

\put(70,40){\circle{2}}

\multiput(50,10)(0,20){4}{\circle{2}}
\multiput(90,10)(0,20){4}{\circle{2}}
\multiput(110,10)(0,20){4}{\circle*{2}}
\multiput(110,10)(0,20){4}{\line(1,0){20}}
\multiput(130,10)(0,20){4}{\circle*{2}}
\put(150,40){\circle{2}}

%\put(45,52){\tiny $e_{n-4}$}

\put(66,-7){\footnotesize $Z''_n$}

\multiput(210,10)(0,20){4}{\circle*{2}}
\multiput(210,10)(0,20){4}{\line(-1,0){15}}
\multiput(210,10)(0,20){4}{\line(1,0){40}}

\put(230,40){\circle*{2}}
\put(230,40){\line(-2,-3){20}}
\put(230,40){\line(-2,-1){20}}
\put(230,40){\line(-2,1){20}}
\put(230,40){\line(-2,3){20}}

\put(230,40){\line(2,-3){20}}
\put(230,40){\line(2,-1){20}}
\put(230,40){\line(2,1){20}}
\put(230,40){\line(2,3){20}}

\put(270,40){\line(-2,-3){20}}
\put(270,40){\line(-2,-1){20}}
\put(270,40){\line(-2,1){20}}
\put(270,40){\line(-2,3){20}}

\put(270,40){\circle*{2}}

\put(274,38){\tiny $e_{n-3}$}

\multiput(250,10)(0,20){4}{\circle*{2}}
\multiput(290,10)(0,20){4}{\circle{2}}
\multiput(310,10)(0,20){4}{\circle*{2}}
\multiput(310,10)(0,20){4}{\line(1,0){20}}
\multiput(330,10)(0,20){4}{\circle*{2}}
\put(350,40){\circle{2}}

\put(266,-7){\footnotesize $W_n$}
\end{picture}

{\bf Figure 6}
\end{center}
\end{figure}

\section{Proposition \ref{prop 1.2}}

In this section, we prove Proposition \ref{prop 1.2}. For the reader's convenience, we rewrite it here:

\begin{prop} \label{prop 4.1}
For $m\ge 2$, the complexes $I(\Delta^m_1), \cdots, I(\Delta^m_4)$ are described as follows:
$$I(\Delta^m_1) = *, I(\Delta^m_2) \simeq S^0, I(\Delta^m_3) \simeq S^1 \vee S^{m-1}, I(\Delta^m_4) \simeq S^m$$
\end{prop}

\begin{proof}
Note that $I(\Delta^m_1)$ is a point. It clearly follows from Proposition \ref{prop Engstrom} that $I(\Delta^m_2) \simeq I(K_2) = S^0$.

Consider the case of $n = 3$. By Lemma \ref{lem 3.1}, we have that $I(\Delta^m_3) \simeq I(X_3)$. Braun and Hough determined the homotopy types of the independence complexes of $X_3$ (see Lemma 3.2 of \cite{BH}), but we give an alternative proof of this result for self-containedness. First Proposition \ref{prop Engstrom} implies that $I(X_3 \setminus e_3)$ and $I(X_3 \setminus \{ e_1, e_3\})$ are homotopy equivalent. Since $X_3 \setminus \{ e_1, e_3\}$ is the $m$-copies of $K_2$, we have 
$$I(X_3 \setminus e_3) \simeq I(X_3 \setminus \{ e_1,e_3\}) = S^{m-1}.$$
On the other hand, applying Proposition \ref{prop Engstrom} again, we have that $I(X_3 \setminus N_{X_3}[e_3])$ and $I(K_2) = S^0$ are homotopy equivalent. Since every map from $S^0$ to $S^{m-1}$ is null-homotopic, the inclusion $I(X_3 \setminus N_{X_3}[e_3]) \hookrightarrow I(X_3 \setminus e_3)$ is null-homotopic. Thus Proposition \ref{prop 2.3} implies $I(X_3) = S^1 \vee S^{m-1}$.

Finally we consider the case $n=4$. By Proposition \ref{prop 3.2} and $I(\Delta^m_1) = *$, we have that $I(X_4) \simeq I(Y_4)$. By Proposition \ref{prop Engstrom}, $I(Y_4 \setminus e_4)$ is homotopy equivalent to the independence complex of the disjoint union of one isolated vertex and $m$-copies of $K_2$, and hence contractible. In particular, the inclusion $I(Y_4 \setminus N_{Y_4}[e_4]) \hookrightarrow I(Y_4 \setminus e_4)$ is null-homotopic, and hence Proposition \ref{prop 2.3} implies $I(Y_4) \simeq \Sigma I(Y_4 \setminus N_{Y_n}[e_4])$. Since $Y_4 \setminus N_{Y_n}[e_4] \cong X_3 \setminus e_3$, we have that $I(Y_4 \setminus N[e_4]) = S^{m-1}$ by the previous paragraph. Thus we conclude that
$$I(\Delta^m_4) \simeq I(Y_4) \simeq \Sigma I(Y_4 \setminus N[e_4]) = S^m.$$
This completes the proof.
\end{proof}

Therefore the complexes $I(\Delta^m_1), \cdots, I(\Delta^m_4)$ are wedges of spheres. Thus Theorem \ref{thm 1.1} implies that all of $I(\Delta^m_n)$ are wedges of spheres and their integral homology groups have no torsions. This gives an answer to a question suggested in the end of Braun and Hough \cite{BH}.

\subsection*{Acknowledgements}
The author thanks an anonymous referee for useful comments which improved the manuscript. The author is supported by JSPS KAKENHI 19K14536.

\end{document}